\documentclass[reqno]{amsart}
\usepackage{amsmath,amsthm,amssymb}
\usepackage{latexsym}
\usepackage{eucal}
\usepackage{fullpage}

\newtheorem{theorem}{Theorem}[section]
\newtheorem{lemma}{Lemma}[section]
\newtheorem{corollary}{Corollary}[section]

\newtheorem{proposition}{Proposition}[section]

\theoremstyle{definition}

\newtheorem{open problem}{Open Problem}

\def\cal{\mathcal}
\let\Re=\undefined
\DeclareMathOperator{\Re}{Re}
\let\Im=\undefined
\DeclareMathOperator{\Im}{Im}

\begin{document}
\title[The Sobolev norms and localization \ldots
 ]{The Sobolev norms and localization on the Fourier side for solutions
 to some evolution equations }
\author{Sergey A. Denisov}
\address{
\begin{flushleft}
University of Wisconsin--Madison\\  Mathematics Department\\
480 Lincoln Dr., Madison, WI, 53706, USA\\
  denissov@math.wisc.edu
\end{flushleft}
  }
 \maketitle
\begin{abstract}
In this paper, some evolution equations with rough time-dependent
potential are studied in the case of one-dimensional torus. We show
that the solution has higher regularity for the generic values of
the coupling parameter. We also control the localization of these
solutions on the Fourier side.
\end{abstract} \vspace{1cm}
\Large

\section{Introduction}
Let $P(x)$ be an algebraic polynomial with real-valued
time-dependent coefficients
\[
P(x)=\sum_{j=1}^d p_jx^j, \quad p_j(t)\in L^1_{\rm loc}(\mathbb{R})
\]
One can consider the following evolution equation
\begin{equation}\label{gen-evol}
iu_t=(kP(i\partial_x)+V)u, \quad u(x,0,k)=1, \quad x\in \mathbb{T},
\quad k\in \mathbb{R}
\end{equation}
where the potential $V(x,t)$ satisfies
\begin{equation}\label{l2}
\|V(x,t)\|_{L^\infty(\mathbb{T})}\in L^1_{\rm loc}(\mathbb{R}^+)
\end{equation}
Since $p_j\in\mathbb{R}$, the unperturbed evolution (i.e. when
$V=0$) defines a unitary group in $L^2(\mathbb{T})$. The assumption
(\ref{l2}) allows one to iterate the Duhamel formula (see
\cite{simon2}) and show that the resulting series converges in
$L^2(\mathbb{T})$. That implies the $L^2(\mathbb{T})$ norm of the
solution is bounded for any $t$ however it might grow as
$t\to\infty$. One question we want to address in this paper is what
happens to the Sobolev norms? Are they bounded for $t>0$ and, if so,
how fast can they grow as $t\to\infty$?

The case of real-valued $V$ is a very special one and we will mostly
focus on that situation. Indeed, if $V\in \mathbb{R}$, then
$\|u(x,t,k)\|_{L^2(\mathbb{T})}=1$ for all $t$. Let $\widehat f_n$
denote the Fourier transform of $f(x)$ in the variable $x\in
\mathbb{T}$. For real-valued $V$, we will study the localization of
solution on the Fourier side and its asymptotical behavior for large
time. In particular, the following question is quite natural: if
initially $\widehat u_n(0,k)=\delta_0$, then what can be said about
the size of
\[
\sum_{|n|>\mu(t)} |\widehat u_n(t,k)|^2
\]
for various $\mu(t)$? If this sum is small, then a nontrivial
$\ell^2$ norm of $\widehat u$ should be supported on the first
$\mu(t)$ frequencies because the total $\ell^2$ norm is conserved
and is equal to $1$. This problem is related to the estimates on the
Sobolev norms but is not equivalent to it. The results we obtain in
this paper answer some of these questions. What makes our setting
different from the earlier extensive work on the subject of
large-time behavior of evolution equations (see, e.g., \cite{ir} and
references there) is that we want to address these problems not for
a particular $k$ but for its ``generic" value with respect to the
Lebesgue measure. The current paper is a
 continuation of \cite{Den1} where the analogous questions were studied mostly
 by the complex analysis technique. In the proofs that follow, we develop more robust
 perturbation theory. For example, we can handle
 equations in which the parameter $k$ enters in a more complicated way, e.g.
 in (\ref{gen-evol}), instead of $k$ we can write $\lambda(k)$ where $\lambda$
  is smooth but not necessarily analytic. The main motivation to study these problems
  comes from the scattering theory
  of multidimensional Schr\"odinger operator with slowly decaying potential
  as explained in \cite{den2}.  For the case of smooth
$V$ the nontrivial upper estimates for the growth of Sobolev norms
were obtained in \cite{b1} where the Schr\"odinger evolution was
handled and the arithmetic structure was used to gain extra
regularity of solution. In the second section, the  Schr\"odinger
evolution with real $V$ is considered. The third section also
handles Schr\"odinger evolution but with general $V$. In the last
section, an asymptotical result for the non-degenerate $P$, i.e.
when $p_1\neq 0$, is obtained. The Appendix contains two auxiliary
lemmas.

For simplicity, we will study only the case of quadratic polynomials
$P(x)$, i.e., $d=2$, however the methods can be easily adjusted to
other symbols. The symbol $\|f\|$ will refer to the
$L^2(\mathbb{T})$ norm in case of a function $f$ and $\|O\|$ will
refer to the operator norm if $O$ is an operator. We will denote the
Hilbert-Schmidt norm of the operator $O$ by $\|O\|_{\cal{S}_2}$. If
$p\in [1,\infty]$, the symbol $p'$ denotes the conjugate exponent,
i.e.
\[
\frac 1p+\frac{1}{p'}=1
\]
For real $\alpha$, $[\alpha]$ stands for the integer part, $f*g$
denotes the convolution of $f$ and $g$, $\chi_A$ is the
characteristic function of the set $A$. For the norms in Sobolev
spaces we have
\[
\|f\|_{H^\alpha(\mathbb{T})}^2=\sum_{n} (1+|n|)^{2\alpha} |\widehat
f_n|^2, \quad \|f\|_{\dot{H}^\alpha(\mathbb{T})}^2=\sum_{n\neq 0}
(1+|n|)^{2\alpha} |\widehat f_n|^2
\]
For two operators $A$ and $B$, the commutator $[A,B]=AB-BA$.
\bigskip

\section{The Schr\"odinger evolution with real-valued $V$}

The important special case  of (\ref{gen-evol}) is the Schr\"odinger
evolution which corresponds to $P(x)=p_2x^2$. We first assume  that
$p_2=1$ so (\ref{gen-evol}) takes the form
\begin{equation}\label{evol}
iu_t=(-k\Delta+V)u, \quad u(x,0)=1, \quad x\in \mathbb{T}, \quad
k\in \mathbb{R},\quad \Delta=\partial^2_{xx}
\end{equation}
%Assume without loss of generality that
%\[
%\int_\mathbb{T} V(x,t)dx=0
%\]
%for all $t$.
Let
\[
\upsilon_1(t)=\|V(x,t)\|_{L^\infty(\mathbb{T})}
\]
and $w(t)\in C[0,T]$ be arbitrary positive function. Take
\[
D_1(T)=\int_0^T \upsilon_1^2(\tau)\left(1+\int_0^\tau
\upsilon_1(\tau_1)d\tau_1\right)d\tau
\]
and
\[
D_2(T)=\left(\int_0^T \upsilon_1^2(\tau)w(\tau)d\tau\right)\left(
\int_0^T\upsilon_1^2(\tau)\int_0^\tau w^{-1}(\tau_1)d\tau_1d\tau
\right)
\]

\begin{theorem}\label{t11}
Suppose $V$ is real-valued and $\alpha<1/2$. Then
\[
\int_{\mathbb{R}} \sup_{\tau\in
[0,T]}\|u(x,\tau,k)\|^2_{\dot{H}^\alpha(\mathbb{T})}dk \lesssim
D_1(T)+D_2(T),\quad \forall T>0
\]
\end{theorem}
In particular, if $V$ is bounded on $\mathbb{T}\times \mathbb{R}^+$
then for Lebesgue a.e. $k$ the solution is $H^\alpha$ regular for
any $t$ and the norm does not grow faster than $t^{1.5+\epsilon}$
with any fixed $\epsilon>0$. We expect much stronger result to hold
and state it as an

{\bf Open problem.} Prove that
\[
\int_{\mathbb{R}} \sup_{\tau\in
[0,T]}\|u(x,\tau,k)\|^2_{\dot{H}^1(\mathbb{T})}dk \lesssim
1+\int_0^T\int_\mathbb{T} V^2(x,\tau)dxd\tau, \quad \forall T>0
\]
\smallskip
We will use the following notation
\begin{equation}\label{volna}
\widetilde V=e^{-ik\Delta t}Ve^{ik\Delta t}
\end{equation}
Take any interval $S\subseteq [0,T]$ and define the operator
$\widetilde V_S(k)$ by its matrix representation on the Fourier side
\[
\widehat{\widetilde V_S(k)}(m,n)=\int_S e^{ik(m^2-n^2) \tau}
\widehat V_{m-n}(\tau) d\tau, \quad m,n\in \mathbb{Z}
\]
%By multiplying with $\chi_{|m|\neq |n|}$, we ignore the diagonals
%$|m|=|n|$ as these sites will be related to interaction between
%eigenmodes $e^{\pm ijx}$ that represent the same eigenvalue $-j^2,
%j=0,1,\ldots$ of the Laplacian.
Let $P_N$ be a projection to the first $N$ Fourier modes
\begin{equation}\label{proe}
\widehat{P_Nf}(n)=\chi_{|n|\leq N}\cdot \widehat f(n)
\end{equation}
and $Q_N=I-P_N$.
\begin{lemma}\label{lema1}
We have
\begin{equation}\label{l11}
\int_\mathbb{R}  \sup_{S}\|P_N\widetilde V_S(k)Q_N\|_{\cal{S}_2}^2
dk\lesssim N^{-1}\log N\int_0^T \int_{\mathbb{T}} V^2(x,t) dxdt
\end{equation}
\end{lemma}
\begin{proof}
For any $S$,
\[
 \sum_{|m|\leq N} \sum_{|n|> N} \left|\int_S \widehat
V_{m-n}(t)e^{ik(m^2-n^2)t}dt\right|^2\leq \sum_{|n|<N} \sum_{|m|\geq
N} |q_{m-n}(k(m^2-n^2))|^2
\]
where
\[
q_l(k)= \sup_{S}\left|\int_S \widehat V_l(t)e^{ikt}dt\right|
\]
By Carleson's theorem on maximal functions \cite{Carl}, we have
\[
v_l=\|q_l\|_{2}\lesssim \|\widehat V_l\|_{2}
\]
The l.h.s. in (\ref{l11}) is bounded by $T_1+T_2$, where
\[
T_1=\sum_{|m|\leq N}\sum_{n> N} \frac{v_{m-n}^2}{n^2-m^2}
\]
and
\[
T_2=\sum_{|m|\leq N}\sum_{n<-N} \frac{v_{m-n}^2}{n^2-m^2}
\]
Since $T_1=T_2$, we only estimate $T_1$. If $\alpha=n-m, \beta=m+n$,
then
\begin{equation}\label{two-terms}
T_1\lesssim  \sum_{\alpha>2N} \frac{v^2_\alpha}{\alpha}
\sum_{\beta=\alpha-2N}^{\alpha+2N}
\frac{1}{\beta+1}+\sum_{\alpha=1}^{2N}
\frac{v^2_\alpha}{\alpha}\sum_{\beta=2N-\alpha}^{2N+\alpha}
\frac{1}{\beta+1}\lesssim  N^{-1}\sum_{\alpha=1}^\infty
v_\alpha^2+\sum_{\alpha=N}^{3N} \frac{v_\alpha^2}{\alpha}\log
\frac{|2N+\alpha|}{|2N-\alpha|+1}
\end{equation}
and that finishes the proof.
\end{proof}
{\bf Remark.} The logarithmic factor in the estimate above is not
present when the Laplacian $\Delta$ is restricted to the Hardy space
$\cal{H}^2(\mathbb{T})$. It is also negligible when the average in
$N$ is taken. Indeed, we have
\[
\sum_{N=1}^\infty\sum_{\alpha=N}^{3N} \frac{v_\alpha^2}{\alpha}\log
\frac{|2N+\alpha|}{|2N-\alpha|+1}\lesssim \sum_{\alpha=1}^\infty
\frac{v_\alpha^2}{\alpha} \int_{\alpha/3}^\alpha\log
\frac{|2x+\alpha|}{|2x-\alpha|+1}dx\lesssim \sum_{\alpha=1}^\infty
v_\alpha^2
\]
so the second term in (\ref{two-terms}), rather than the first one,
is in $\ell^1$.\smallskip

Now, we are ready to prove theorem \ref{t11}. The technique will
resemble the one used in \cite{kis} for the matrices $2\times 2$.

\begin{proof}{\it (of theorem \ref{t11})}
Consider $\phi_n=P_nu, \psi_n=Q_nu$. Then,
\[
i\partial_t\psi_n=(-kQ_n\Delta Q_n +Q_nVQ_n)\psi_n+Q_nV\phi_n, \quad
\psi_n(x,0,k)=0
\]
and so
\begin{equation}\label{nno}
\|\psi_n\|^2= 2\Im \int_0^t \langle Q_nV\phi_n,\psi_n\rangle d\tau
\end{equation}
If $\phi_n=e^{ik\Delta t}\widetilde \phi_n$  and $\psi_n=e^{ik\Delta
t}\widetilde \psi_n$, then
\begin{equation}\label{deriv}
i\partial_t \widetilde \psi_n=(Q_n\widetilde V Q_n)\widetilde
\psi_n+(Q_n\widetilde V P_n) \widetilde \phi_n, \quad i\partial_t
\widetilde \phi_n=(P_n\widetilde V P_n)\widetilde
\phi_n+(P_n\widetilde V Q_n) \widetilde \psi_n
\end{equation}
and integration by parts in (\ref{nno}) yields
\[
\|\psi_n(t,k)\|^2\lesssim  I_1+I_2
\]
where
\[
I_1=\Im \int_0^t \langle \Bigl( Q_n\widetilde V_{[\tau,
t]}P_n\Bigr)\widetilde \phi_n', \widetilde\psi_n\rangle d\tau,
\,I_2=\Im \int_0^t \langle \Bigl( Q_n\widetilde V_{[\tau,
t]}P_n\Bigr)\widetilde\phi_n, \widetilde\psi_n'\rangle d\tau
\]
For $I_1$, we have
\[
I_1\lesssim \int_0^t \|Q_n\widetilde V_{[\tau,t]}P_n\|\cdot
\|V\|\cdot \|\psi_n\|d\tau
\]
due to (\ref{deriv}) and $\|\phi_n\|\leq 1, \|\psi_n\|\leq 1$.

 For $I_2$, we substitute (\ref{deriv}) to get
\begin{eqnarray}
I_2\lesssim \int_0^t \|Q_n\widetilde V_{[\tau,t]}P_n\|\cdot
\|V\|\cdot \|\psi_n\|d\tau+\nonumber
\\
\left|\Re \int_0^t \langle \left(\int_\tau^t Q_n\widetilde V(\tau_1)
P_nd\tau_1\right) \widetilde \phi_n(\tau), (Q_n\widetilde
V(\tau)P_n) \widetilde \phi_n(\tau)\rangle d\tau \right|
\label{vtor}
\end{eqnarray}
For the last term, we can use the following identity
\[
2\Re \int_0^t \langle Z'(\tau)y(\tau),Z(\tau)y(\tau)\rangle d\tau=
\|Z(\tau)y(\tau)\|^2\Bigl|_{\tau=0}^{\tau=t}-2\Re \int_0^t \langle
Z(\tau)y'(\tau), Z(\tau)y(\tau)\rangle d\tau\Bigr.
\]
Thus, the second term in (\ref{vtor}) is bounded by
\[
a_1(t,k)=\|Q_n \widetilde V_{[0,t]} P_n\|^2+\int_0^t \|Q_n
\widetilde V_{[\tau,t]} P_n\|^2\|V(\tau)\|d\tau
\]
If we denote
\[
a_2(t,k)=\int_0^t \|Q_n\widetilde V_{[\tau,t]}P_n\|\cdot \|V\| d\tau
\]
and $z(k)=\max_{\tau\in [0,T]} \|\psi_n\|=\|\psi_n(t(k),k)\|$, then
the quadratic inequality
\[
z^2\lesssim a_2z+a_1
\]
gives
\begin{equation}
z\lesssim a_2+\sqrt{a_1}
\end{equation}
In other words,
\begin{eqnarray*}
\sup_{\tau\in [0,T]}\sum_{|j|>n} |\widehat u_j(\tau,k)|^2\lesssim
a_2^2(t(k),k)+a_1(t(k),k)\lesssim \\\left(\int_0^T
w(\tau)\|V(.,\tau)\|_{L^\infty(\mathbb{T})}^2d\tau\right)\left(\int_0^T
w^{-1}(\tau)\sup_{S\subseteq [\tau,T]}\|Q_n \widetilde V_{S}
P_n\|^2d\tau\right)+a_1(t(k),k)
\end{eqnarray*}
where we applied Cauchy-Schwarz. By lemma \ref{lema1}, we have
\[
\int_\mathbb{R}  \sup_{\tau\in [0,T]}\sum_{|j|>n} |\widehat
u_j(\tau,k)|^2dk\lesssim \frac{\log n}{n}(D_1+D_2)
\]
where $D_{1(2)}$ were introduced above. Multiply the last estimate
by $|n|^{-\epsilon}$ and sum in $n\neq 0$ to get theorem \ref{t11}.
\end{proof}
Assume that $\alpha<1/2, \gamma>3/4$ is fixed and
$\upsilon_1(t)\lesssim (1+t)^{-\gamma}$.  Take $w(t)=(1+t)^{1/2}$.
Then, we have the following striking estimate
\[
\sup_{t>0}\|u\|_{H^\alpha(\mathbb{T})}<\infty
\]
for a.e. $k$. This is a remarkable fact as we do not assume any
smoothness of $V$ at all.
\bigskip

In the rest of this section, we will consider the problem which is
directly related to the multidimensional scattering \cite{den2}. We
again take $P(x)=c_2x^2$ but now the coefficient decays in $t$
\[
c_2(t)=\frac{1}{(1+t)^2}
\]
The difficult problem in this area is to show that the solution has
localization $\mu(t)\leq t$ for a.e. $k$ as longs as potentials $V$
satisfies some decay condition, e.g.
\[
|V(x,t)|<C(1+t)^{-1/2-\epsilon}, \quad \epsilon\in (0,1/2)
\]
In fact, this is not known even for $\epsilon$ close to $1/2$.

Below we will give a partial solution to this problem in the case
when $V$ oscillates. This will improve on the earlier result from
\cite{Den1}. The following proof is general enough
 to handle initial data of the form $u(x,0,k)=e^{ijx}$ for
 any $j$ and thus it yields that the whole monodromy matrix for (\ref{gen-evol}) is
 ``almost diagonal".

\begin{theorem}
Suppose that $V$ can be written as $V=Q_x(x,t)/(t+1)$ where $Q$ is
real valued and
\begin{equation}\label{con2}
\|Q\|_{L^\infty(\mathbb{T})}<\lambda(t+1)^{-\gamma},\,
\|Q_x\|_{L^\infty(\mathbb{T})}<\lambda(t+1)^{1-\gamma}, \quad
\gamma>3/4
\end{equation}
Let $V_T=V\cdot \chi_{t>T}$ and $u_T$ be corresponding solution.
Then, we have
\[
\int_\mathbb{R}\sup_{t>T}\|1-u_T(x,t,k)\|^4dk \lesssim \lambda^4
(T+1)^{3-4\gamma}
\]
\end{theorem}
\begin{proof}
We will suppress the dependence of $u$ on $T$ and will write $u$
instead of $u_T$. Let $\widehat u_n(t,k)=e^{-in^2k/(t+1)}\widehat
{\widetilde u}_n(t,k)$. For the zero Fourier mode of $u$, we have
\[
\widehat u_0(t,k)=1+\sum_{|j|\geq 1}^\infty \int_T^t \frac{j\widehat
Q_j(\tau)}{\tau+1}e^{ikj^2/(\tau+1)}\widehat{\widetilde
u}_j(\tau,k)d\tau
\]
Integration by parts gives
\[
|\widehat u_0-1|\lesssim \left|\sum_{|j|\geq 1}^\infty \int_T^t
\left(\int_\tau^t \frac{j\widehat Q_j(\tau_1)}{\tau_1+1}e^{ikj^2
/(\tau_1+1)}d\tau_1\right)\widehat{\widetilde u'}_j(\tau,k)d\tau\right|
\]
From Cauchy-Schwarz and $\|\partial_t \widetilde u\|\lesssim
\lambda(t+1)^{-\gamma}$, we have
\begin{eqnarray*}
\sup_{t>T}|\widehat u_0-1|^2\lesssim \hspace{10cm}
\\\lambda^2\left(\int_T^\infty (\tau+1)^{-1-\epsilon}d\tau
\right)\left(\int_T^\infty
(\tau+1)^{-2\gamma+1+\epsilon}\sum_{|j|\geq 1}^\infty
\left(\sup_{I\subseteq [\tau,\infty)}\left| \int_I \frac{j\widehat
Q_j(\tau_1)}{\tau_1+1}e^{ikj^2 /(\tau_1+1)}d\tau_1
\right|\right)^2d\tau\right)
\end{eqnarray*}
and the Carleson theorem implies (after the change of variables
$\xi=(\tau_1+1)^{-1}$ for the integral in $\tau_1$)
\[
\int_\mathbb{R}\sup_{t>T}|\widehat u_0-1|^2dk\lesssim
\lambda^4(T+1)^{3-4\gamma}
 \]
Then, notice that
\begin{equation}\label{trick1}
\sum_{j} |\widehat u_j|^2=1,
\quad 1-|\widehat u_0|^2\leq 2 |1-\widehat u_0|
\end{equation}
and then
\[
\|1-u\|^2= |1-\widehat u_0|^2+\sum_{j\neq 0} |\widehat
u_j|^2\lesssim |1-\widehat u_0|
\]
This estimate finishes the proof.
\end{proof}
This theorem immediately implies for $T=0$ and small $\lambda$ that
for the positive measure set of $k$ the solution has a zero mode
bounded away from origin for all time. It also implies the
localization on the Fourier
 side  but only  in the regime of small
  $\lambda$. \bigskip

\section{The Schr\"odinger flow with complex-valued  $V$}

In this section, we again study (\ref{evol}) but we do not assume
that $V$ is real-valued and thus the $L^2$ norm of the solution is
not necessarily conserved. However, for generic $k$, not only the
$L^2$ norm will be bounded but also the Sobolev norms
$H^\alpha(\mathbb{T})$ where $\alpha<1$. This is an improvement on
the results from the second section however the bound will be
exponential in time. The proof will be based on the concept of the
variation norm as advocated in \cite{VN} with ideas going back to
\cite{tl}. The key result from \cite{VN} we will use is the
following estimate (see \cite{VN}, $(68)$, Appendix B): if ${\cal
P}=\{\Delta_j\}, \Delta_j=[t_j,t_{j+1})$ is any partition of
$\mathbb{R}$, then
\begin{equation}\label{varya}
\left\|\sup_{\cal P} \left(\sum_j \left|\int_{t_j}^{t_{j+1}}
f(t)e^{ikt}dt\right|^2\right)^{1/2}\right\|_{L^{p'}(\mathbb{R})}\lesssim
\|f\|_p, \quad p\in [1,2)
\end{equation}
We first restrict the problem (\ref{evol}) to the case of finite
matrices. Instead of (\ref{evol}), consider
\begin{equation}\label{comple}
X_t=e^{-ik\Delta t}V^{(N)}e^{ik\Delta t} X, \quad
X(0,k)=I_{(2N+1)\times (2N+1)}
\end{equation}
where $V^{(N)}=P_{N}VP_{N}$ so we will be dealing with matrices of
the finite size but the estimates we obtain must be independent of
$N$. Assume first that the time $t\in [0,1]$, we will handle the
intervals $[0,T]$ by scaling later.

 Let us go to the Fourier side
and then
\[
\widehat X_t=\widetilde V^{(N)}\widehat X, \quad \widehat X(0,k)=I
\]
where $\widetilde V$ is given by (\ref{volna}). Consider the scale
$\ell^{2,\alpha}$ of the weighted $\ell^2$ spaces with the norm
\[
\|f\|_{2,\alpha}=\left(\sum_{|j|\leq N}
|f_j|^2(1+|j|)^{2\alpha}\right)^{1/2}
\]
If $O$ is a linear operator in $\mathbb{C}^{2N+1}$, we will denote
its operator norm in $\ell^{2,\alpha}$ by $\|O\|_{\alpha}$. On the
group $\cal{G}=GL(2N+1, \mathbb{C})$, consider the following metric
\[
d_{\cal{G}}(A,B)=\inf_{\gamma}\int_0^1 \|\gamma'\gamma^{-1}\|_\alpha
dt
\]
where $\gamma(t)$ is any continuously differentiable path in
$\cal{G}$ such that $\gamma(0)=A$ and $\gamma(1)=B$. Here we assume
that both $A$ and $B$ lie in the same connected component of
$GL(2N+1, \mathbb{C})$.\bigskip

{\bf Remark.} Let $\gamma$ be any curve such that
\[
\int_0^1 \|\gamma'\gamma^{-1}\|_\alpha dt<d_{\cal{G}}(A,B)+1
\]
Then
\[
\gamma(t)=A+\int_0^t \gamma'(\tau)\gamma^{-1}(\tau)\gamma(\tau)d\tau
\]
and therefore
\[
\|\gamma(t)\|_\alpha\leq
\|A\|_\alpha+\int_0^t\|\gamma'(\tau)\gamma^{-1}(\tau)\|_\alpha\|\gamma(\tau)\|_\alpha
d\tau
\]
By Gronwall-Bellman, we have
\begin{equation}\label{metr}
\|B\|_\alpha\lesssim \|A\|_\alpha \exp\left(\int_0^1
\|\gamma'(\tau)\gamma^{-1}(\tau) \|_\alpha d\tau\right)\lesssim
\|A\|_\alpha \exp\left(d_{\cal{G}}(A,B)\right)
\end{equation}
\bigskip

\begin{theorem}\label{coole}
Suppose $V(x,t)\in L^\infty(\mathbb{T}\times [0,1])$ and $X(t,k)$ is
the solution to (\ref{comple}) on the interval $[0,1]$. Let
$\alpha\in(0,1)$ and $p\in (4/3,2)$ so that $\alpha p<2(p-1)$. Then,
we have
\[
\sup_{0<t<1}\log (1+\|\widetilde X(t,k)\|_\alpha)\lesssim
1+\|V\|_\infty^{1+s_0}+\|V\|_\infty^{1-s^2_0}U^{s_0(1+s_0)}(k)
\]
where
\[
s_0=\alpha p'/2 \quad {\rm and}\quad \|U(k)\|_{p'}\lesssim
\|V\|_\infty=\sup_{x\in \mathbb{T},t\in [0,1]}|V(x,t)|
\]
\end{theorem}
\begin{proof}
 Recall \cite{VN} that for the continuous curve
$\gamma(t)$ on $\cal{G}$ we can define
\[
\|\gamma\|_{V^\beta}=\sup_{{\cal P}}\left( \sum_{j=0}^{n-1}
d^\beta_{\cal{G}}(\gamma(t_{j}),\gamma(t_{j+1})) \right)^{1/\beta},
\quad \beta\in [1,\infty)
\]
where $\cal P$ is any partition. Then,
we have (\cite{VN}, lemma~C.3)
\begin{equation}\label{gam}
\|\gamma\|_{V^\beta}\leq
\|\gamma_r\|_{V^\beta}+C\min(\|\gamma_r\|^2_{V^\beta},
\|\gamma_r\|^\beta_{V^\beta})
\end{equation}
where
\[
\gamma_r(t)=\int_0^t \gamma'(s)\gamma^{-1}(s)ds
\]
and $\beta\in [1,2)$. Thus, (\ref{comple}), (\ref{metr}),
(\ref{gam}) and the simple estimate (that follows from the
definition of the variation norm)
\[
d_{\cal{G}}(\gamma(0),\gamma(1))\leq \|\gamma\|_{V^\beta}
\]
imply

\begin{equation}\label{vare}
\sup_{t\in [0,1]} \|\widehat X(t,k)\|_{\alpha}\lesssim
\exp(Q+C\min\{Q^2,Q^\beta\}), \, Q=\left\|\int_0^t \widetilde
V(\tau,k)d\tau\right\|_{V^\beta [0,1](\ell^{2,\alpha})}
\end{equation}
and thus we only need to obtain a bound for $Q$. Let
$\Lambda^\gamma$ be a diagonal operator on the Fourier side with the
diagonal elements equal to $(2+|n|)^{\gamma},\,\gamma\in
\mathbb{R}$.

To handle $\beta\in (1,2)$, we will use the standard complex
interpolation between $\beta=1$ and $\beta=2$. Take $s_0\in (0,1)$
and $\mu\in (0,1)$. Suppose we fix a partition $\cal{P}$ of the interval
$[0,1]$. Then, for any $s\in (0,1)$, take $\beta(s)=s+1,
q(s)=(s+1)/s$. The $\ell^\beta$ norm can be written as
\begin{eqnarray*}
\left(\sum_{j}\left\|\int_{t_j}^{t_{j+1}}  \Lambda^{s\mu} \widetilde
V(\tau,k)
\Lambda^{-s\mu}d\tau\right\|^{\beta(s)}\right)^{1/\beta(s)}=\hspace{8cm}
\\
\max_{\|\eta\|_{\ell^{q(s_0)}}=1}\sum_{j}\left\|\int_{t_j}^{t_{j+1}}
\Lambda^{s\mu} \widetilde V(\tau,k)
\Lambda^{-s\mu}d\tau\right\||\eta_j|^{q(s_0)/q(s)}
\end{eqnarray*}
since
\[
\frac{1}{q(s)}+\frac{1}{\beta(s)}=1
\]
For the norm of the operator, we have another variational
representation
\begin{eqnarray*}
\left\|\int_{t_j}^{t_{j+1}}  \Lambda^{s\mu} \widetilde V(\tau,k)
\Lambda^{-s\mu}d\tau\right\|=\hspace{8cm}
\\
\max_{\|f\|_{\ell^2}=\|g\|_{\ell^2}=1}\left|\int_{t_j}^{t_{j+1}}
\sum_{|m|,|n|\leq N}(2+| m|)^{s\mu} \widetilde
V_{m,n}(\tau,k)(2+|n|)^{-s\mu}f_mg_nd\tau\right|
\end{eqnarray*}
One can arrange the maximizers $f', g'$ such that the last sum is
equal to
\[
\int_{t_j}^{t_{j+1}} \sum_{m,n}(2+| m|)^{s\mu} \widetilde
V_{m,n}(\tau,k)(2+|n|)^{-s\mu}f'_mg'_nd\tau
\]
i.e. the absolute value can be dropped. Thus, we only need to bound
\[
F(s)=\sum_j |\eta_j|^{q(s_0)/q(s)}\int_{t_j}^{t_{j+1}}\sum_{m,n}
(2+| m|)^{s\mu} \widetilde
V_{m,n}(\tau,k)(2+|n|)^{-s\mu}f'^{(j)}_mg'^{(j)}_nd\tau
\]
where $\|\eta\|_{\ell^{q(s_0)}}=1$ and
$\|f_j'\|_{\ell^2}=\|g_j'\|_{\ell^2}=1$ for all $j$.

 Notice that $F(s)$ is analytic in $s$ in the strip $0<\Re s<1$
and we can apply the three lines lemma there \cite{folland}.
\[
M_0=\sup_{\Re s=0} |F(s)|\leq \sum_j (t_{j+1}-t_j)\sup_{t\in
[t_j,t_{j+1}] }\|V(t)\|\lesssim \|V\|_{L^\infty(\mathbb{T}\times
[0,1])}
\]
and
\[
M_1=\sup_{\Re s=1} |F(s)|\leq \left(\sum_j  |\eta_j
|^{q(s_0)}\right)^{1/2}  \left( \sum_j \left\|\int_{t_j}^{t_{j+1}}
\Lambda^\mu \widetilde V(\tau,k) \Lambda^{-\mu}d\tau
\right\|^2\right)^{1/2}
\]
\[
\leq \|\int_0^t \widetilde V\|_{V^2(\ell^{2,\mu})}
\]
By the three line lemma, we have
\begin{equation}\label{interp}
|F(s_0)|\leq M_0^{1-s_0}M_1^{s_0}=\|V\|_\infty^{1-s_0}\|\int_0^t
\widetilde V\|_{V^2(\ell^{2,\mu})}^{s_0}
\end{equation}
Taking the supremum over all partitions, we have the standard
interpolation
\begin{equation}\label{iitp}
\|\widetilde V\|_{V^{s_0+1}(\ell^{2,s_0\mu})}\leq
\|V\|_\infty^{1-s_0}\|\int_0^t \widetilde
V\|_{V^2(\ell^{2,\mu})}^{s_0}
\end{equation}
for any $s_0\in [0,1]$. Next, we will focus on the bounds for the
second variation norm because it enters into (\ref{iitp}).
\begin{lemma}
Suppose  $p\in(4/3,2)$ and $\mu=2/p'$. Then, we have
\[
\left(\sup_{\cal{P}} \sum_j \|\Lambda^{\mu}\widetilde
V_{\Delta_j}(k)\Lambda^{-\mu}\|^2 \right)^{1/2}\lesssim
\|V\|_\infty+U(k)
\]
where
\begin{equation}\label{l12}
\|U(k)\|_{p'}\lesssim \|V\|_\infty, \quad \widetilde
V_{\Delta_j}(k)=\int_{\Delta_j} \widetilde V(\tau,k)d\tau
\end{equation}
\end{lemma}
\begin{proof}
We have
\[
\Lambda^{\mu}\widetilde V_{\Delta_j}(k)\Lambda^{-\mu}=\widetilde
V_{\Delta_j}(k)+[\Lambda^\mu, \widetilde V_{\Delta_j}(k)]
\Lambda^{-\mu}
\]
For the first term, we have an obvious estimate
\begin{equation}\label{plain1}
\sum_j \|\widetilde V_{\Delta_j}(k)\|^2\leq \|V\|_\infty^2 \sum_j
|\Delta_j|^2\leq \|V\|_\infty^2
\end{equation}
For the second one,
\[
\sup_{\cal{P}} \sum_j \|[\Lambda^{\mu},\widetilde
V_{\Delta_j}(k)]\Lambda^{-\mu}\|_{\cal{S}_2}^2\lesssim \sum_{m,n}
\left(\frac{|m|^\mu-|n|^\mu}{1+|n|^\mu}\right)^2 \sup_{\cal{P}}
\sum_j \left|\int_{\Delta_j} \widehat
V_{m-n}(t)e^{ik(m-n)(m+n)t}dt\right|^2
\]
Let $\alpha=m-n$ and $\beta=m+n$. Notice that
\[
||m|^\mu-|n|^{\mu}|\sim ||m|-|n||\cdot ||m|+|n||^{\mu-1}
\]
Therefore, we have two terms to bound
\[
I_1=\sum_{|\alpha|>1}\sum_{|\beta|\geq |\alpha|} \left(
\frac{\alpha |\beta|^{\mu-1}}{1+|\beta-\alpha|^\mu}
\right)^2\sup_{\cal{P}} \sum_j \left|\int_{\Delta_j} \widehat
V_{\alpha}(t)e^{ik\alpha\beta t}dt\right|^2
\]
and
\[
I_2=\sum_{|\alpha|>1}\sum_{|\beta|<|\alpha|}\left(
\frac{|\alpha|^{\mu-1} |\beta|}{1+|\beta-\alpha|^\mu}
\right)^2\sup_{\cal{P}} \sum_j \left|\int_{\Delta_j} \widehat
V_{\alpha}(t)e^{ik\alpha\beta t}dt\right|^2
\]
Now, (\ref{varya}) implies
\[
\left\|I_2
\right\|_{L^{p'/2}(\mathbb{R})}\lesssim\sum_{\alpha>1}\sum_{|\beta|<\alpha}\left(
\frac{\alpha^{\mu-1} |\beta|}{1+|\beta-\alpha|^\mu} \right)^2
 \left\|\sup_{\cal{P}} \sum_j
\left|\int_{\Delta_j} \widehat V_\alpha(t)e^{ik\alpha\beta
t}dt\right|^2\right\|_{L^{p'/2}(\mathbb{R})}
\]
\[
\lesssim\sum_{\alpha>1}\sum_{|\beta|<\alpha}\left(
\frac{\alpha^{\mu-1} |\beta|}{|\alpha\beta|^{1/p'}(
1+|\beta-\alpha|^\mu)} \right)^2
 \left(\int_0^1 |\widehat V_\alpha(t)|^pdt \right)^{2/p}
\]
\[
\leq \sum_{\alpha>1}\sum_{|\beta|<\alpha}\left( \frac{\alpha^{\mu-1}
|\beta|}{|\alpha\beta|^{1/p'}(1+|\beta-\alpha|^\mu)} \right)^2
 \left(\int_0^1 |\widehat V_\alpha(t)|^2dt \right)\lesssim
 \]
\[
 \sum_{\alpha>1} \|V_\alpha\|_2^2 (\alpha^{1-4/p'}+\alpha^{2\mu-4/p'})\lesssim \|V\|_\infty^2
 \]
as long as
\[
\frac 12<\mu\leq \frac{2}{p'}
\]
For $I_1$, the estimate is similar
\[
\left\|I_1
\right\|_{L^{p'/2}(\mathbb{R})}\lesssim\sum_{\alpha>1}\sum_{|\beta|\geq
\alpha}\left( \frac{\alpha |\beta|^{\mu-1}}{1+|\beta-\alpha|^\mu}
\right)^2
 \left\|\sup_{\cal{P}} \sum_j
\left|\int_{\Delta_j} \widehat V_\alpha(t)e^{ik\alpha\beta
t}dt\right|^2\right\|_{L^{p'/2}(\mathbb{R})}
\]
\[
\leq \sum_{\alpha>1}\sum_{|\beta|\geq \alpha}\left( \frac{\alpha
|\beta|^{\mu-1}}{|\alpha\beta|^{1/p'}(1+|\beta-\alpha|^\mu)}
\right)^2
 \left(\int_0^1 |\widehat V_\alpha(t)|^2dt \right)\lesssim
 \]
\[
 \sum_{\alpha>1} \|V_\alpha\|_2^2 (\alpha^{2\mu-4/p'}+\alpha^{1-4/p'})\lesssim \|V\|_\infty^2
 \]
Combining these bounds with (\ref{plain1}), we have the statement of
the lemma.
\end{proof}
The lemma gives a necessary bound for the variation norm so
(\ref{vare}) and (\ref{iitp}) then finish the proof of the theorem.
\end{proof}
The immediate corollary of this theorem is
\begin{lemma}\label{reshe}
Under the conditions of the theorem \ref{coole},  assume that $u$ is the
solution to
\[
iu_t=\left(k\Delta+V\right)u, \quad u(x,0,k)=1
\]
Then, we have
\[
\sup_{0<t<1}\log (1+\|u(.,t,k)\|_{H^\alpha(\mathbb{T})})\lesssim
1+\|V\|_\infty^{1+s_0}+\|V\|_\infty^{1-s^2_0}U^{s_0(1+s_0)}(k)
\]
where
\[
\|U(k)\|_{p'}\lesssim \|V\|_\infty
\]
\end{lemma}
\begin{proof}
Consider $u^{(N)}$ which solves
\[
iu^{(N)}_t=\left(kP_{N}\Delta+V^{(N)}\right)u^{(N)}, \quad
u^{(N)}(x,0,k)=1
\]
By approximating lemma (the lemma 4.1 in \cite{Den1}, which also
works in our setting), we have
\[
\sup_{t\in [0,1],\,k\in [-A,A]}\|u^{(N)}(x,t,k)-u(x,t,k)\|\to 0,
\quad N\to \infty
\]
for any fixed $A$. Therefore, given any fixed $m$, we have
\[
\sup_{0<t<1}\log (1+\| P_mu(x,t,k)\|_\alpha)\lesssim
1+\|V\|_\infty^{1+s_0}+\|V\|_\infty^{1-s^2_0}U^{s_0(1+s_0)}(k)
\]
with $m$-independent $U$. Taking $m\to\infty$, we have the statement
of the lemma.

\end{proof}
We conclude this section with
\begin{theorem}
If $V\in L^\infty(\mathbb{T}\times [0,T])$ and
\[
iu_t=\left(k\Delta+V\right)u, \quad u(x,0,k)=1
\]
then (under the conditions of the theorem \ref{coole})
\[
\sup_{0<t<T}\log (1+\|u(.,t,k)\|_{H^\alpha(\mathbb{T})})\lesssim
1+(T\|V\|_\infty)^{1+s_0}+(T\|V\|_\infty)^{1-s^2_0}U_1^{s_0(1+s_0)}(k)
\]
and
\[
\|U_1(k)\|_{p'}\lesssim T^{1-1/p'}\|V\|_\infty
\]
\end{theorem}
\begin{proof}
It is sufficient to notice that
$\psi(x,\tau,\kappa)=u(x,T\tau,\kappa/T)$ solves the problem
\[
i\psi_\tau=\kappa \Delta \psi+TV(x,T\tau)\psi, \quad
\psi(x,0,\kappa)=1
\]
for $\tau\in [0,1]$ and rescale using the lemma \ref{reshe}.
\end{proof}
The following corollary is immediate

{\bf Remark.} Let $\alpha<1$ and $V\in L^\infty(\mathbb{T}\times
[0,\infty))$. We have
\begin{equation}\label{eee}
\sup_{t>0}\frac{\|u(x,t,k)\|_{H^\alpha(\mathbb{T})}}{\exp(t^{2})}\leq
C(k,\alpha)
\end{equation}
for a.e. $k\in \mathbb{R}$. The simple example of $V(x,t)=i$ shows
that the exponential growth is possible even for $L^2(\mathbb{T})$
norm. It is likely that $\exp(t^2)$ can be replaced by $\exp(t^\mu),\, \mu>1$.
\bigskip

\section{The case of real $V$, small gaps, and the nondegenerate symbol}

In this section, we will prove the localization result for a
particular  case of (\ref{gen-evol}) when the  symbol is
nondegenerate, i.e.
\[
P'(0)=1\neq 0
\]
and there is a small $\delta>0$ such that the equation
\begin{equation}\label{secsec1}
P(x)=E
\end{equation}
has the unique solution $x_E$ for every $E\in
(P(0)-\delta,P(0)+\delta)$. The last condition is satisfied, e.g.,
when $P(x)$ is strictly monotonic.

 We will also assume that $c_j(t)=\alpha_j(t+1)^{-j}$ and
$\alpha_j$ are constants. This is the hard case when the gaps
between the eigenvalues of the differential operator are decreasing
in $t$. That represents the real difficulty in the analysis of the
multidimensional scattering \cite{den2}.

In this paper, we will consider the quadratic polynomial only which
leads to
\begin{equation}\label{gaps}
iu_t=k\left(\frac{i\partial_x}{t+1}-P_{[t/2]}\frac{\partial^2_{xx}}{(t+1)^2}\right)u+Vu,
\quad u(x,0,k)=1
\end{equation}
Notice that we introduced the Fourier projection in the second term
to make sure that the second condition (i.e., \eqref{secsec1}) is
satisfied.

Another example for which the method works is
\begin{equation}\label{gaps1}
iu_t=ik\left(\frac{\partial_x}{t+1}-\frac{\partial^3_{xxx}}{(t+1)^3}\right)u+Vu,
\quad u(x,0,k)=1
\end{equation}

We will obtain the asymptotical result as $t\to\infty$ for the $u$
with the standard WKB-type correction coming from the corresponding
transport equation. This will done under the assumption that $V$ is
real and decays like $t^{-\gamma}$ with $\gamma<1$ being very close
to $1$.

The following lemma is quite standard
\begin{lemma} \label{lemma-d1} Suppose $O_1(t)$ is an operator-valued function
such that $\|O_1\|\in L^1[0,\infty)$. Consider two equations
\[
i\partial_t \psi_1=O\psi_1,\quad \psi_1(0)=f_1
\]
and
\[
i\partial_t \psi_2=(O+O_1)\psi_2+j, \quad \psi_2(0)=f_2
\]
where $O(t), O_1(t)$ are both self-adjoint and
\[
\|j\|, \|O\|, \|O_1\|\in L^1_{\rm loc}(\mathbb{R}^+)
\]
Then,
\[
\sup_{t\in [0,T]} \|\psi_1-\psi_2\|\leq \int_0^T
(\|O_1(t)\|+\|j(t)\|)dt+\|f_2-f_1\|
\]
\end{lemma}
\begin{proof}
Let $W_1$ be the solution to
\[
i\partial_t W_1(t_1,t)=O(t)W_1(t_1,t),\quad  W_1(t_1,t_1)=I
\]
Since $O$ is self-adjoint, the operator $W_1$ is unitary. Then, by
Duhamel's formula, we have
\[
\psi_2(t)=W_1(0,t)f_2-i\int_0^t
W_1(\tau,t)(O_1(\tau)\psi_2(\tau)+j(\tau))d\tau
\]
Since $W_1(0,t)f_1=\psi_1$, $\|\psi_2\|=1$, and $W_1$ is unitary,
\[
\|\psi_2(t)-\psi_1(t)\|\leq \|f_2-f_1\|+\int_0^t
(\|O_1(\tau)\|+\|j(\tau)\|)d\tau
\]
\end{proof}
This lemma will allow us to throw away any $L^1$ perturbations with
small norm when the localization question is studied.\smallskip

Restrict the problem (\ref{gaps}) to the diadic intervals $I=[T,2T]$
first. Then, on the Fourier side, we can write
\begin{eqnarray}\label{rest}
i \widehat u_t=\left(k\left(\frac{n}{T}
P_{[T^\alpha]}+Q_{[T^\alpha]}\left(\frac{n}{t+1}+P_{[t/2]}\frac{n^2}{(t+1)^2}\right)\right)+
\widehat V* +V_1\right)\widehat u\\
\widehat u(T,k)=\widehat u^{(0)}\nonumber
\end{eqnarray}
where $P_{[T^\alpha]}$ and $Q_{[T^\alpha]}$ are defined in
(\ref{proe}),
\[
V_1=kP_{[T^\alpha]}\frac{n^2}{(t+1)^2}+kP_{[T^\alpha]}\left(\frac{n}{t+1}-\frac{n}{T}\right)
\]
and
\begin{equation}\label{reduc}
\sup_{t\in I, k\in [-A,A]} \|V_1\|\lesssim T^{2\alpha-2}
\end{equation}
for every fixed $A$.

We take
\begin{equation}\label{al1}
\alpha<1/2
\end{equation}
to make sure that $\int_I \|V_1\|dt\sim T^{2\alpha-1}$ is
small.\smallskip

Consider the following problem now
\begin{eqnarray}\label{diadic}
i  \phi_t=\left(k\left(\frac{n}{T}
P_{[T^\alpha]}+Q_{[T^\alpha]}\left(\frac{n}{t+T+1}+P_{[(t+T)/2]}\frac{n^2}
{(t+T+1)^2}\right)\right)+\widehat V*\right) \phi,\\
\, \phi(x,0,k)=\phi^{(0)}(x,k)\nonumber
\end{eqnarray}
It is obtained from (\ref{rest}) by dropping $V_1$ (the error made
by doing that will be taken care of by lemma \ref{lemma-d1}) and by
shifting to the time interval $t\in [0,T]$.

Denote by $\nu=\exp(i\mu(x,t,k))$ the solution to the transport
equation
\[
i\nu_t=ik\nu_x/T+V\nu, \quad \nu(x,0,k)=1
\]
\begin{theorem}\label{diad}
Let $\gamma\in (21/22,1)$, $|V(x,t)|<C(t+1)^{-\gamma}$, and
$\phi^{(0)}$ satisfies the following properties
\begin{equation}\label{betar}
\sup_{k\in [-A,A]}\|\phi^{(0)}\|_{L^\infty(\mathbb{T})}\leq 1, \quad
\sup_{k\in [-A,A]}\|\phi^{(0)}\|_{H^{1/2}(\mathbb{T})}\lesssim
T^{1.5(1-\gamma)}
\end{equation}
where $A>1$ is fixed. Then, we have
\begin{equation}\label{oce-1}
\int_{-A}^{A} \sup_{t\in [0,T]}
\|\phi(x,t,k)-\nu(x,t,k)\phi^{(0)}(x+kt/T,k)\|^4dk\lesssim
T^{-(1-\gamma)}
\end{equation}
\end{theorem}
\begin{proof}
Write (\ref{diadic}) in the block form
\begin{equation}\label{blok}
i\partial_t\left(
\begin{array}{c}
\phi_1\\
\phi_2
\end{array}
\right) = \left[
\begin{array}{cc}
k\Lambda_{1}+V_{11}& V_{12} \\
V_{21} & k\Lambda_2+V_{22}
\end{array}
\right]\left(
\begin{array}{c}
\phi_1\\
\phi_2
\end{array}
\right)
\end{equation}
where $\displaystyle
k\Lambda_1+V_{11}=P_{[T^\alpha]}\left(k\frac{n}{T}+\widehat
V*\right)P_{[T^\alpha]}$. Notice that this operator is a restriction
of the transport equation to the first $T^\alpha$ modes so we start
with proving localization results for this operator.

\begin{lemma}\label{auxil}
Let $f(x,k)$ be such that
\[
\|f(x,k)\|_{L^\infty(\mathbb{T})}\leq 1,\quad
\|f(x,k)\|_{H^{1/2}(\mathbb{T})}\lesssim T^\beta
\]
and  $y$ solve the problem
\[
iy_t=(k\Lambda_1+V_{11})y, \quad y(x,0,k)=P_{[T^\alpha]} f(x,k)
\]
Then, we have the following representation
\[
y(x,t,k)=\nu(x,t,k)f(x+kt/T,k)+\delta(x,t,k)
\]
and
\begin{equation}\label{l2-e}
\int_{-A}^A \sup_{t\in [0,T]}\|\delta(x,k,t)\|^2dk\lesssim
T^{2-2\gamma+2\beta-\alpha}+T^{4-4\gamma-\alpha}
\end{equation}
\end{lemma}
\begin{proof}
Notice that $\mu$ is real-valued. In (\cite{Den1}, lemma 2.1) we
proved that
\[
\int_{\mathbb{R}} \sup_{t\in [0,T]}\sum_{j} |j||\widehat
\mu_j(t,k)|^2dk \lesssim T\int_0^T\int_{\mathbb{T}} V^2(x,t)dxdt
\lesssim T^{2-2\gamma}
\]
That implies (\cite{Den1}, Appendix, lemma 6.2) that $\nu$ is
unimodular and
\[
\int_{-A}^A \sup_{t\in [0,T]}
\|\nu(x,t,k)\|^2_{H^{1/2}(\mathbb{T})}dk\lesssim T^{2-2\gamma}
\]
for any fixed $A$. Thus $\nu^{(1)}(x,t,k)=\nu(x,t,k) f(x+kt/T,k)$
solves the transport equation with initial data $f(x,k)$ and
\[
\|\nu^{(1)}\|_{L^\infty(\mathbb{T})}\leq 1
\]
From the second lemma in Appendix, we have
\[
\|\nu^{(1)}\|_{H^{1/2}(\mathbb{T})}\lesssim
T^\beta+\|\nu\|_{H^{1/2}(\mathbb{T})}
\]
Therefore
\begin{equation}\label{inter-m}
\int_{-A}^A  \sup_{t\in [0,T]} \sum_{|n|>T^\delta}|\widehat
\nu^{(1)}_n(t,k)|^2dk\lesssim
T^{-\delta+2-2\gamma}+T^{2\beta-\delta}
\end{equation}
for every $\delta>0$.

The function $\nu^{(2)}=P_{[T^\alpha]}\nu^{(1)}$ satisfies
\[
i\nu^{(2)}=(k\Lambda_1+V_{11})\nu^{(2)}+\delta^{(1)},\quad
\nu^{(2)}(x,0,k)=P_{[T^\alpha]}f(x,k)
\]
where
\[
\delta^{(1)}=P_{[T^\alpha]}VQ_{[T^\alpha]}\nu^{(1)}
\]
and so (\ref{inter-m}) gives
\begin{equation}\label{cut1}
\int_{-A}^A  \sup_{t\in [0,T]} \|\delta^{(1)}\|^2dk\lesssim
T^{-2\gamma}(T^{-\alpha+2-2\gamma}+T^{2\beta-\alpha})
\end{equation}
The  lemma \ref{lemma-d1} yields
\[
\int_{-A}^A  \sup_{t\in [0,T]} \|y-\nu^{(2)}\|^2dk\lesssim
T^{2-2\gamma}(T^{-\alpha+2-2\gamma}+T^{2\beta-\alpha})
\]
By (\ref{inter-m}),
\[
\int_{-A}^A  \sup_{t\in [0,T]} \|\nu^{(1)}-\nu^{(2)}\|^2dk\lesssim
T^{2\beta-\alpha}+T^{2-2\gamma-\alpha}
\]
This gives (\ref{l2-e}).
\end{proof}
Let us choose $\beta=1.5(1-\gamma)$ as in (\ref{betar}) and consider
the solution to (\ref{blok}). Notice first that
\[
\|Q_{[T^\alpha]}\phi^{(0)}\|\lesssim T^{\beta-\alpha/2}
\]
so, since evolution preserves the $L^2(\mathbb{T})$ norm, we have
\begin{equation}\label{col1}
\sup_{t\in [0,T]}\|\phi(x,t,k)-\zeta(x,t,k)\|\lesssim
T^{\beta-\alpha/2}
\end{equation}
where $\zeta$ solves (\ref{diadic}) with initial condition
$\zeta(x,0,k)=\zeta^{(0)}=P_{[T^\alpha]}\phi^{(0)}$, i.e.
\begin{equation}\label{blokada}
i\partial_t\left(
\begin{array}{c}
\zeta_1\\
\zeta_2
\end{array}
\right) = \left[
\begin{array}{cc}
k\Lambda_{1}+V_{11}& V_{12} \\
V_{21} & k\Lambda_2+V_{22}
\end{array}
\right]\left(
\begin{array}{c}
\zeta_1\\
\zeta_2
\end{array}
\right), \quad \zeta_1(0,k)=P_{[T^\alpha]}\phi^{(0)}, \zeta_2(0,k)=0
\end{equation}

Let $W(t,k)$ be defined as
\[
i\partial_t W=e^{ik\Lambda_1 t}V_{11}e^{-ik\Lambda_1 t} W, \quad
W(0,k)=I
\]
If
\[
\zeta_1=e^{-ik\Lambda_1 t}W \widetilde \zeta_1, \quad
\zeta_2=e^{-ik\Lambda_2 t}\widetilde \zeta_2
\]
then
\begin{equation}\label{blokada1}
i\partial_t\left(
\begin{array}{c}
\widetilde \zeta_1\\
\widetilde \zeta_2
\end{array}
\right) = \left[
\begin{array}{cc}
0 & W^{-1}e^{ik\Lambda_1 t}V_{12}e^{-ik\Lambda_2 t} \\
e^{ik\Lambda_2 t}V_{21}e^{-ik\Lambda_1 t}W & e^{ik\Lambda_2
t}V_{22}e^{-ik\Lambda_2 t}
\end{array}
\right]\left(
\begin{array}{c}
\widetilde \zeta_1\\
\widetilde \zeta_2
\end{array}
\right)
\end{equation}
For $\widetilde \zeta_1$,
\[
\widetilde \zeta_1(t,k)=\zeta^{(0)}-i\int_0^t
W^{-1}(\tau,k)e^{ik\Lambda_1\tau}V_{12}(\tau)e^{-ik\Lambda_2
\tau}\widetilde \zeta_2(\tau,k)d\tau
\]
We take an inner product with $\zeta^{(0)}$ and integrate by parts
\begin{equation}\label{sec-ter}
\langle \widetilde \zeta_1,\zeta^{(0)}\rangle
=\|\zeta^{(0)}\|^2-i\int_0^t \langle
e^{ik\Lambda_1\tau}V_{12}(\tau)e^{-ik\Lambda_2 \tau}\widetilde
\zeta_2(\tau,k), W(\tau,k)\zeta^{(0)}\rangle d\tau
\end{equation}
As $\|\widetilde \zeta_1\|^2+\|\widetilde
\zeta_2\|^2=\|\zeta^{(0)}\|^2$, it is sufficient to show that the
second term in (\ref{sec-ter}) is small to guarantee that
$\widetilde \zeta_1=\zeta^{(0)}+{\rm ``small"}$ and $\widetilde
\zeta_2={\rm ``small"}$ (see lemma \ref{o-triv} from Appendix).

 Let
\[
\widetilde V_{12}(t,k)=e^{ik\Lambda_1 t} V_{12}(t) e^{-ik\Lambda_2
t}
\]
and write
\[
\int_0^t \langle \widetilde V_{12}(\tau,k)\widetilde
\zeta_2(\tau,k), \widetilde y(\tau,k)\rangle d\tau=I_1+I_2
\]
Here
\[
\widetilde y(t,k)=W(t,k)\zeta^{(0)},\, I_1=\int_0^t \langle
\widetilde \zeta_2, \widetilde V_{12}^*P_{[T^\delta]} \widetilde
y\rangle d\tau, \, I_2=\int_0^t \langle \widetilde \zeta_2,
\widetilde V_{12}^*Q_{[T^\delta]} \widetilde y\rangle d\tau
\]
where $\delta$ will be chosen later. From the lemma \ref{auxil}, we
have
\[
\int_{-A}^A \sup_{t\in [0,T]}\| Q_{[T^\delta]} \widetilde
y(t,k)\|_2^2 dk\lesssim
(T^{2-2\gamma}+T^{2\beta})T^{-\delta}+T^{2-2\gamma+2\beta-\alpha}+T^{4-4\gamma-\alpha}
\]
since
\[
\int_{-A}^A \sup_{t\in [0,T]} \|Q_{[T^\delta]}\Bigl(\nu
f\Bigr)\|^2dk\lesssim T^{-\delta}(T^{2\beta}+T^{2-2\gamma})
\]

 Therefore, by Cauchy-Schwarz,
\[
\left(\int_{-A}^A \sup_{t\in [0,T]}|I_2(t,k)|^2
dk\right)^{1/2}\lesssim
T^{1-\gamma}\left((T^{1-\gamma}+T^{\beta})T^{-\delta/2}+T^{1-\gamma+\beta-\alpha/2}+
T^{2-2\gamma-\alpha/2}\right)\sim T^{-\epsilon_1}
\]
and
\[
\epsilon_1=\min\{\delta/2-2(1-\gamma), \delta/2-\beta-(1-\gamma),
\alpha/2-\beta-2(1-\gamma), \alpha/2-3(1-\gamma)\}
\]
which yields the following conditions
\[
2(1-\gamma)<\delta/2,\, 1-\gamma+\beta<\delta/2,\,
2(1-\gamma)<\alpha/2-\beta,\, 3(1-\gamma)<\alpha/2
\]
For $I_1$, we integrate by parts and use $\widetilde \zeta_2(0,k)=0$
to get
\[
I_1=\int_0^t \langle \widetilde \zeta_2', Q(\tau,t,k) \widetilde
y\rangle d\tau+\int_0^t \langle \widetilde \zeta_2, Q(\tau,t,k)
\widetilde y'\rangle d\tau
\]
and
\[
Q(\tau,t,k)=\int_\tau^t V_{12}^*(\tau_1,k)P_{[T^\delta]} d\tau_1
\]
For the derivatives, we have
\[
\|\widetilde \zeta'_2\|\lesssim T^{-\gamma}, \,\|\widetilde
y'\|\lesssim T^{-\gamma}
\]
and for the Hilbert-Schmidt norm of $Q$,
\[
\sup_{t,\tau\in [0,T]}\|Q\|^2_{\cal{S}_2}\lesssim
\sum_{|l|>T^\alpha}\sum_{|j|<T^\delta} \sup_{\tau\in [0,T]} \left|
\int_\tau^T \widehat
V_{l-j}(\tau_1)\exp\left(ik(l/T+\chi_{|l|<0.5T}l^2/T^2-j/T)\tau_1\right)d\tau_1
\right|^2
\]
The Carleson's theorem on the maximal function again gives
\[
\int \sup_{t,\tau\in [0,T]}\|Q\|^2_{\cal{S}_2} dk\lesssim
T^{1-2\gamma+1-\alpha}\cdot T^\delta
\]
as long as
\begin{equation}\label{cond1}
\delta<\alpha
\end{equation}
That yields an estimate for $I_1$
\[
\left(\int |\sup_{t\in [0,T]} I_1|^2dk\right)^{1/2}\lesssim
T^{2(1-\gamma)}\cdot T^{(\delta-\alpha)/2}
\]
Combining the bounds obtained above, we have
\[
\left(\int_{-A}^A \sup_{t\in [0,T]}|\langle \widetilde
\zeta_1(t,k),\zeta^{(0)}\rangle
-\|\zeta^{(0)}\|^2|^2dk\right)^{1/2}\lesssim T^{-\epsilon}
\]
where
\[
\epsilon=\min\{ \epsilon_1, (\alpha-\delta)/2-2(1-\gamma)\}
\]
so we add one more condition on the parameters
\begin{equation}\label{cond2}
\alpha>\delta+4(1-\gamma)
\end{equation}
Consequently, the lemma \ref{o-triv} from Appendix gives
\[
\int \sup_{t\in [0,T]} \|\zeta(t,k)-e^{-ik\Lambda_1
t}W(0,t,k)\zeta^{(0)}\|^4 dk\lesssim T^{-2\epsilon}
\]
since
\[
\widetilde \zeta_1=W^{-1}e^{ik\Lambda_1 t}\zeta_1
\]
For $e^{-ik\Lambda_1 t}W(0,t,k)\zeta^{(0)}$, the lemma \ref{auxil}
is applicable and we have
\[
\int_{-A}^A \sup_{t\in [0,T]}\|\phi-\nu
\phi^{(0)}(x+kt/T,k)\|^4dk\lesssim
T^{4\beta-2\alpha}+T^{-2\epsilon}+T^{2-2\gamma+2\beta-\alpha}+T^{4-4\gamma-\alpha}
\]
by using (\ref{col1}) and $\|\delta(x,k,t)\|\lesssim 1$. For
simplicity, let us make the following choices of our parameters:
\[
\delta=6(1-\gamma),\, \alpha=11(1-\gamma)
\]
and so
\[
\int_{-A}^A \sup_{t\in [0,T]}\|\phi-\nu \phi^{(0)}\|^4dk\lesssim
T^{-(1-\gamma)}
\]
Since we will need $\alpha<1/2$ (check (\ref{al1})), the condition
$\gamma>21/22$ follows.
\end{proof}
\bigskip

Now, we are ready to prove the similar statement for the problem
(\ref{rest}).
\begin{theorem}\label{diad1}
Let $\gamma\in [83/87,1)$, $|V(x,t)|<C(t+1)^{-\gamma}$, and
$u^{(0)}$ satisfies the following properties
\[
\sup_{k\in [-A,A]}\|u^{(0)}\|_{L^\infty(\mathbb{T})}\leq 1, \quad
\sup_{k\in [-A,A]}\|u^{(0)}\|_{H^{1/2}(\mathbb{T})}\lesssim
T^{1.5(1-\gamma)}
\]
where $A>1$ is fixed. Then, for the solution of (\ref{rest}), we
have
\begin{equation}\label{oce-2}
\int_{-A}^{A} \sup_{t\in [T,2T]}
\|u(x,t,k)-\nu(x,t,k)u^{(0)}(x+kt/T,k)\|^4dk\lesssim T^{-(1-\gamma)}
\end{equation}
\end{theorem}
\begin{proof}
Apply the theorem \ref{diad}, lemma \ref{lemma-d1}, and
(\ref{reduc}) to get
\begin{equation}\label{oce-3}
\int_{-A}^{A} \sup_{t\in [T,2T]}
\|u(x,t,k)-\nu(x,t,k)u^{(0)}(x+kt/T,k)\|^4dk\lesssim
T^{-(1-\gamma)}+T^{4(2\alpha-1)}
\end{equation}
If $\gamma\geq 83/87$, the first term is larger and we have the
statement of the theorem.\bigskip
\end{proof}

Let us define the solution to the following transport equation
\[
iG_t=ikG_x/(1+t)+VG, \quad G(x,0,k)=1
\]
\begin{lemma}
For an arbitrary fixed $A>0$, we have
\begin{equation}\label{strong}
\int_{-A}^A \sup_{t<T}
\|G(x,t,k)\|_{H^{1/2}(\mathbb{T})}^2dk\lesssim \int_0^T
(1+t)\int_\mathbb{T} V^2(x,t)dx dt
\end{equation}
\end{lemma}
\begin{proof}
Consider
\[
H(x,t)=G(x,e^{t}-1)
\]
For $H$, we have
\[
iH_t(x,t,k)=ikH_x(x,t,k)+V_s(x,t)H(x,t,k),\quad H(x,0,k)=1
\]
where
\[
V_s(x,t)=e^tV(x,e^t-1)
\]
From \cite{Den1}, we have
\[
\int \sup_{t\in [0,\tau]} \sum_{j} |j||\widehat
H_j(t,k)|^2dk\lesssim \int_0^\tau \int_\mathbb{T} V_s^2(x,t)dxdt
\]
Taking $\tau=\log T$ with large $T$, we get the statement of the
lemma.
\end{proof}

 Now that we can control the behavior of $u$ on every diadic interval,
we can prove the main result of this section
\begin{theorem} Assume that $\gamma\in [83/87,1)$ and $V$ in (\ref{gaps}) satisfies
\[
\|V(x,t)\|_{L^\infty(\mathbb{T})}\leq \lambda (1+t)^{-\gamma}
\]
Then
\begin{equation}\label{asi}
\int_{-A}^A \sup_{t>0} \|u(x,t,k)-G(x,t,k)\|dk \to 0
\end{equation}
as $\lambda\to 0$. Here $A$ is any positive number.
\end{theorem}
\begin{proof}
Take $T_j=2^j$ and consider the diadic intervals $[T_j,T_{j+1})$. On
any fixed interval $[0,T_j]$ we have
\[
\sup_{t\in [0,T_j]}\|u-1\|\to 0, \sup_{t\in [0,T_j]}\|G-1\|\to 0
\]
when $\lambda\to 0$. This convergence is  uniform in  $k\in [-A,A]$.
Therefore, it is sufficient to assume that we solve the problem on
the interval $[T_N,\infty)$ instead where $N$ is sufficiently large.

 The estimate (\ref{strong}) implies that
\[
\int_{-A}^A \|G(x,T_j,k)\|_{H^{1/2}(\mathbb{T})}^2dk\lesssim
T_j^{2-2\gamma}
\]
Therefore,
\begin{equation}\label{svede}
\int_{-A}^A \sum_j
\frac{\|G(x,T_j,k)\|_{H^{1/2}(\mathbb{T})}^2}{T_j^{2\beta}} dk
<\infty, \quad \beta\in (1-\gamma, 1.5(1-\gamma))
\end{equation}
If $\nu$ solves
\[
i\nu_t=ikq(t)\nu_x+V\nu,\quad  \nu(x,0,k)=1
\]
with $q$ given by: $q(t)=T_j^{-1},\, t\in [T_j, T_{j+1})$, then
\[
\sup_{t>0}\|\nu-G\|\to 0
\]
as $\lambda\to 0$ and similarly
\begin{equation}
\int_{-A}^A \sum_j
\frac{\|\nu(x,T_j,k)\|_{H^{1/2}(\mathbb{T})}^2}{T_j^{2\beta}} dk
<\infty, \quad \beta\in (1-\gamma, 1.5(1-\gamma))
\end{equation}
Then, we can consider $\Omega_\lambda$: the set of those $k\in
[-A,A]$ for which
\[
\sum_j \frac{\|\nu(x,T_j,k)\|_{H^{1/2}(\mathbb{T})}^2}{T_j^{2\beta}}
dk <1
\]
As $\lambda\to 0$, the measure of $\{[-A,A]\backslash
\Omega_\lambda\}$ will converge to $0$. Then, restricting  $k$ to
the set $\Omega_\lambda$, we can recursively apply theorem
\ref{diad1} to show that
\[
\sup_{t>0} \|u-G\|
\]
converges to zero in measure (as a function in $k\in
\Omega_\lambda$) provided that $\lambda\to 0$. This is equivalent to
(\ref{asi}) since $\|u\|=\|G\|=1$.
\end{proof}
{\bf Remark.} It is conceivable that the constant $83/87$ can be
decreased by more efficient choice of parameters. This theorem is
important since it can handle a difficult case of $\gamma<1$ for
very general class of pseudodifferential operators. Indeed, we used
the fact that $V$ is multiplication operator but the polynomial
$P(x)$ can be replaced by other nondegenerate symbols.

\section{Appendix}
We used the following elementary lemma in the main text.

\begin{lemma}\label{o-triv}
Suppose ${\cal V}$ is a vector space with the inner product and $v,
a$ are two vectors such that $v=v_1+v_2$, and
\[
v_1\perp v_2, \quad a\perp v_2, \quad \|v_1\|^2+\|v_2\|^2=\|a\|^2
\]
Then,
\begin{equation}\label{trivia}
\|v-a\|\leq 2\sqrt{|\|a\|^2-\langle v_1,a\rangle|}
\end{equation}
\end{lemma}
\begin{proof}
Assume first that $\|a\|=1$.  We have an orthogonal decomposition
\[
v=v_2+\langle v_1, a\rangle a+\psi
\]
Therefore
\begin{equation}\label{ortog}
\|v_2\|^2+\|\psi\|^2+|\langle v_1,a\rangle|^2=1
\end{equation}
and
\[
\|v-a\|^2=\|v_2\|^2+\|\psi\|^2+|1-\langle v_1,a\rangle|^2
\]
However, (\ref{ortog}) implies
\[
\|v_2\|^2+\|\psi\|^2=1-|\langle v_1,a\rangle|^2
\]
and so we have
\[
\|v-a\|^2\leq 4|1-\langle v_1,a\rangle|
\]
If $\|a\|\neq 1$, rescale the vectors by $\|a\|$ to get
(\ref{trivia}).
\end{proof}

The following lemma is quite standard in the description of Krein's
algebra (e.g., \cite{bosi}, p. 123, formula $(5.2)$ or \cite{simka},
proposition 6.1.10)
\begin{lemma}
Let $f,g\in H^{1/2}(\mathbb{T})\cap L^\infty(\mathbb{T})$, then
\[
\|fg\|_{H^{1/2}(\mathbb{T})}\lesssim \|f\|_{L^\infty(\mathbb{T})}
\|g\|_{H^{1/2}(\mathbb{T})}+\|g\|_{L^\infty(\mathbb{T})}
\|f\|_{H^{1/2}(\mathbb{T})}
\]
\end{lemma}
\begin{proof}
The proof immediately follows, e.g., from
\[
\|f\|^2_{H^{1/2}(\mathbb{T})}\sim
\|f\|^2+\int_\mathbb{T}\int_\mathbb{T} \frac{|f(x)-f(y)|^2}{|x-y|^2}
dxdy
\]
\end{proof}

\section{Acknowledgment}
This research was supported by NSF grant DMS-1067413.\bigskip

\end{document}